\documentclass{amsart}
\newtheorem{Thm}{Theorem}[section]
\newtheorem{Prop}[Thm]{Proposition}
\newtheorem{Cor}[Thm]{Corollary}
\newtheorem{Lem}[Thm]{Lemma}

\theoremstyle{remark}

\theoremstyle{problem}
\newtheorem{Prob}[Thm]{Problem}
\numberwithin{equation}{section}
\begin{document}

\title[Pompeiu problem and discrete groups]
{The Pompeiu problem and discrete groups}

\author[M. J. Puls]{Michael J. Puls}
\address{Department of Mathematics \\
John Jay College-CUNY \\
524 West 59th Street \\
New York, NY 10019 \\
USA}
\email{mpuls@jjay.cuny.edu}
\thanks{The research for this paper was partially supported by PSC-CUNY grant 65364-00 43}

\begin{abstract}
We formulate a version of the Pompeiu problem in the discrete group setting. Necessary and sufficient conditions are given for a finite collection of finite subsets of a discrete abelian group, whose torsion free rank is less than the cardinal of the continuum, to have the Pompeiu property. We also prove a similar result for nonabelian free groups. A sufficient condition is given that guarantees the harmonicity of a function on a nonabelian free group if it satisfies the mean-value property over two spheres. 
\end{abstract}

\keywords{abelian group, free group, harmonic function, left translates, mean-value property, Pompeiu property, radial functions, torsion free rank, zero divisors}
\subjclass[2010]{Primary: 43A15; Secondary: 43A32, 43A70, 43A80, 43A90}

\date{May 3, 2013}
\maketitle

\section{Introduction}\label{Introduction}
Let $\mathbb{C}$ be the complex numbers, $\mathbb{R}$ the real numbers, $\mathbb{Z}$ the integers and $\mathbb{N}$ the natural numbers. Let $2 \leq n \in \mathbb{N}$ and let $\mathcal{K}$ be a finite family of compact subsets of $\mathbb{R}^n$. The Pompeiu problem asks under what conditions on $\mathcal{K}$ will $f = 0$ be the only continuous function on $\mathbb{R}^n$ that satisfies
\begin{equation} 
 \int_{\sigma(K)} f \, dx = 0  \label{eq:contpomp}
\end{equation}
for all $K \in \mathcal{K}$ and for all rigid motions $\sigma$? The answer to this question is no when $\mathcal{K}$ contains only a sphere of positive radius. This led to what is known as the two circle problem, which asks if $\mathcal{K}$ consists of two spheres of positive radius $r_1$ and $r_2$, then is $f = 0$ the only continuous function satisfying (\ref{eq:contpomp})? It turns out that the answer is yes as long as the ratio $r_1/r_2$ avoids a certain exceptional subset of $\mathbb{R}$. See Sections 5 and 6 of the excellent paper \cite{Zalcman80} for the details.

A variant of the Pompeiu problem is the following two circle problem: Let $f$ be a continuous function on $\mathbb{R}^n$, denote by $S_r(x)$ the sphere of radius $r$ centered at $x$. Suppose  
\begin{equation} 
\frac{1}{\mbox{Vol}(S_r(x))} \int_{S_r(x)} f dy = f(x) \label{eq:contmeanvalue}
\end{equation}
for $r = r_1, r_2$ and all $x \in X$. Does this imply $f$ is harmonic on $\mathbb{R}^n$? The answer to this two radius problem is also yes as long as the quotient $r_1/ r_2$ avoids a certain exceptional subset, depending on $n$, of $\mathbb{R}$. For more information see \cite[Section 11]{Zalcman80}. The Pompeiu problem and some of its variations have been studied in various contexts, see \cite{BerensteinZalcman80, CareyKaniuthMoran91, CohenPicardello88, PeyerimhoffSamior10, RawatSitaram97, ScottSitaram88, Zalcman80} and the references therein for more information. The purpose of this study is to investigate (\ref{eq:contpomp}) and (\ref{eq:contmeanvalue}) in the discrete group setting.

For the rest of this paper $G$ will always denote a discrete group. So, the compact subsets of $G$ are precisely the finite subsets of $G$ and all complex-valued functions on $G$ are continuous. Let $\mathcal{K}$ be a finite collection of finite subsets of $G$. We now state the Pompeiu problem for discrete groups with respect to left translations.

\begin{Prob} \label{rightpompeiu}
When is $f = 0$ the only complex-valued function from a given class of functions on $G$ that satisfies
\begin{equation}
 \sum_{x \in gK} f(x) =0 \label{eq:Pompeiu} 
\end{equation}
for all $K \in \mathcal{K}$ and for all $g \in G$?
\end{Prob}

We shall say that $\mathcal{K}$ has the {\em Pompeiu property} with respect to left translations if $f=0$ is the only function from the class of functions under consideration that satisfies (\ref{eq:Pompeiu}). For brevity, we will say that $\mathcal{K}$ has the Pompeiu property when it has the Pompeiu property with respect to left translations. 

Now suppose $G$ is finitely generated with generating set $x_1, \dots ,x_n$. The Cayley graph $\Gamma$ of $G$ is the graph whose vertices are the elements of $G$, and $g, h \in G$ are connected by an edge if and only if $h = gx_i^{\pm 1}$ for some generator $x_i$ of $G$. Note that $\Gamma$ is a connected graph. Now $G$ is a metric space via the shortest path metric on $\Gamma$. That is, if $x, y \in G$, then $d(x,y)$ equals the number of edges of the shortest path in $\Gamma$ joining $x$ and $y$. Let $x \in G$ and let $n \in \mathbb{N}$. The sphere of radius $n$ centered at $x$ is defined to be
\begin{equation}
S_n (x) = \{ w \in G \mid d(x,w) =n \}. \label{eq:defsphere}
\end{equation}
The cardinality of $S_n(x)$ will be denoted by $\vert S_n(x) \vert$. A complex-valued function $f$ on $G$ is said to satisfy the mean-value property for radius $n$ if
\begin{equation}
\frac{1}{\vert S_n(x) \vert} \sum_{w \in S_n(x)} f(w) = f(x)  \label{eq:genmeanvalue}
\end{equation}
for all $x \in G$. We now state a problem that is a discrete group version of (\ref{eq:contmeanvalue}).
\begin{Prob} \label{discretetworadharm}
Suppose $n,m \in \mathbb{N}$. If a complex-valued function $f$ on $G$ satisfies
\[ \frac{1}{\vert S_n(x) \vert} \sum_{w \in S_n(x)} f(w) = f(x) =\frac{1}{\vert S_m(x)\vert} \sum_{w \in S_m(x)} f(w) \]
for all $x \in G$, then is $f$ harmonic?
\end{Prob}

In Section \ref{preliminaries} we give needed definitions and prove some preliminary results. In particular we make a connection between the problem of zero divisors on torsion free discrete groups and Problem \ref{rightpompeiu}. We will give some examples of groups for which (\ref{eq:Pompeiu}) is true when we assume an extra condition on the decay of $f$. We state our main results in Section \ref{statementmainresults}. The torsion free rank of an abelian group $G$, which we shall denote by $r_0(G)$, is the cardinality of a maximal independent system of elements in $G$ of infinite order. Indicate the cardinal of the continuum by $2^{\omega}$. In Section \ref{abelianpompeiu} we prove Theorem \ref{abelianchar}, which gives necessary and sufficient conditions on a finite collection of finite subsets of $G$ to have the Pompeiu property when $G$ is abelian and $r_0(G) < 2^{\omega}$. Let $F_k$ be the free group on $k$ generators, where $k \geq 2$. In Section \ref{freegroupspompeiu} we prove Theorem \ref{pompeiufreechar}, which gives necessary and sufficient conditions on a finite set of finite radial subsets, to be defined in Section \ref{statementmainresults}, of $F_k$ to have the Pompeiu property. Furthermore, in Section \ref{freegroupspompeiu} we will explain how our results for $F_k$ overlap with some of the results in \cite{CohenPicardello88}, where two circle problems are studied on homogenous trees. In Section \ref{meanvalue} we study Problem \ref{discretetworadharm} for free groups. In particular, we give a sufficient condition which guarantees that the harmoncity of a function $f$ on $F_k$ is equivalent to it having the mean-value property over two spheres.

I would like to thank Peter Linnell for making many useful comments on an earlier version of the paper.

\section{Preliminaries and the problem of zero divisors}\label{preliminaries}
For $z \in \mathbb{C}, \bar{z}$ will denote the complex conjugate of $z$. We represent a complex-valued function $f$ on $G$ by a formal sum $f = \sum_{g\in G} a_g g$ where $a_g \in \mathbb{C}$ and $f(g) = a_g$. Also, we shall write $\tilde{f}$ for $\sum_{g\in G} a_g g^{-1}$ and $\bar{f}$ for $\sum_{g \in G} \overline{a_g} g$. Indicate by $\mathcal{F}(G)$ the set of all functions on $G$. For $1 \leq p \in \mathbb{R}$, $\ell^p(G)$ will consist of those formal sums for which $\sum_{g \in G} \vert a_g \vert^p < \infty$, and $\mathbb{C}G$, the group ring of $G$ over $\mathbb{C}$, will consist of all formal sums that satisfy $a_g = 0$ for all but finitely many $g$. The group ring $\mathbb{C}G$ can also be thought of as the set of all functions on $G$ with compact support. For $g \in G$ and $f \in \mathcal{F}(G)$, the left translate $L_g(f)$ of $f$ is defined to be $(L_g(f))(x) = f(gx)$. 

Let $\alpha = \sum_{g \in G} a_gg \in \mathbb{C}G$ and $f = \sum_{g\in G} b_g g \in \mathcal{F}(G)$. Define a map $\langle \cdot, \cdot \rangle \colon \mathbb{C}G \times \mathcal{F} \to \mathbb{C}$ by
\[ \langle \alpha, f \rangle = \sum_{g \in G} a_g\overline{b_g}. \]
For a fixed $h \in \mathcal{F}(G), \langle \cdot, h \rangle$ is a linear functional on $\mathbb{C}G$. Now let $T$ be a linear functional on $\mathbb{C}G$. Define $f(g) = \overline{T(g)}$ for each $g \in G$. Thus each linear functional on $\mathbb{C}G$ defines an element of $\mathcal{F}(G)$. Hence, the dual of $\mathbb{C}G$ can be identified with $\mathcal{F}(G)$. If $E$ is a subset of $G$, denote by $\chi_{E}$ the characteristic function on $E$. Let $g \in G$. Observe that $L_g(\chi_E) = \chi_{g^{-1}E}$. Suppose $\sum_{x \in gK} f(x) = 0$ for all $g \in G$ and for some finite subset $K$ of $G$. Then
\[ \sum_{x \in gK} f(x) = \sum_{x \in G} \chi_{gK}(x)f(x) = \sum_{x \in G} (L_{g^{-1}}(\chi_K))(x) f(x) = \langle (L_{g^{-1}}(\chi_K)), \bar{f}\rangle. \]

We have just proved
\begin{Lem} \label{translates}
Let $K$ be a finite subset of $G$, and let $f$ be a function on $G$. Then $\sum_{x \in gK}f(x) = 0$ for all $g \in G$ if and only if $\langle L_g(\chi_K), \bar{f} \rangle = 0$ for all $g \in G$.
\end{Lem}
Lemma \ref{translates} will eventually allow us to give a link between Pompeiu's problem and the problem of zero divisors.

Let $X$ represent either $\mathcal{F}(G), \ell^p(G)$ or $\mathbb{C}G$. For $\alpha = \sum_{g \in G} a_g g \in \mathbb{C}G$ and $f = \sum_{g \in G} b_g g \in 
X$, we define a multiplication, known as convolution, $\mathbb{C}G \times X \to X$ by 
\[ \alpha \ast f = \sum_{g, h \in G} a_g b_h gh = \sum_{g \in G} \left( \sum_{h\in G} a_{gh^{-1}} b_h \right) g. \]
\begin{Lem} \label{zerodivtrans} 
Let $\alpha \in \mathbb{C}G$ and let $f \in \mathcal{F}(G)$. Then 
   $\alpha \ast \tilde{f} = 0$ if and only if $\langle L_g(\alpha), \bar{f} \rangle = 0$ for all $g \in G$.
  \end{Lem}

\begin{proof}
Set $\alpha = \sum_{g \in G}a_g g$ and $f = \sum_{g \in G} b_g g$. Note that for $g \in G, L_g(\alpha) = \sum_{x \in G} a_{gx} x$. Then
\[ \alpha \ast \tilde{f} = \sum_{g \in G} \left( \sum_{x \in G} a_{gx}b_x \right) g \]
and $\langle L_g(\alpha), \bar{f} \rangle = \sum_{x \in G} a_{gx} b_x$. The result now follows. 
\end{proof}

By combining Lemma \ref{translates} and Lemma \ref{zerodivtrans} we obtain
\begin{Prop} \label{pompeiuzerodiv}
Let $\mathcal{K}$ be a finite collection of finite subsets of $G$. Then $f = 0$ is the only function on $G$ that satisfies (\ref{eq:Pompeiu}) if and only if $f = 0$ is the only solution to the system of convolution equations
\[ \chi_K \ast \tilde{f} = 0  \mbox{ for all } K \in \mathcal{K}. \]
\end{Prop} 

We just saw that there is a connection between Pompeiu's problem on discrete groups and the following problem concerning zero divisors.

\begin{Prob}
Let $G$ be a torsion free group and let $X$ be one of the following: $\mathbb{C}G, \ell^p(G)$ or $\mathcal{F}(G)$. If $0 \neq \alpha \in \mathbb{C}G$ and if $0 \neq f \in X$, then is $\alpha \ast f = 0?$
\end{Prob}

We point out that if $G$ is not torsion free, then there are zero divisors in $\mathbb{C}G$. For example, let 1 be the identity element of $G$ and let $g \in G$ such that $g \neq 1$ and $g^n = 1$ for some $n \in \mathbb{N}$. Then $(1 + g + \cdots + g^{n-1}) \ast (1 - g) = 0$. It now follows that the set $K = \{ 1, g, \dots, g^{n-1}\}$ does not have the Pompeiu property. To see this, set $ f = 1-g \in \mathbb{C}G$ and note $\chi_K = 1 + g + \cdots + g^{n-1}$. Using the fact $g^{-k} = g^{n-k}$ for $0 \leq k \leq n$, we have $ \chi_K \ast \tilde{f} =  0$. By Proposition \ref{pompeiuzerodiv}, $f = 1-g$ satisfies (\ref{eq:Pompeiu}) when $K = \{ 1, g, \dots, g^{n-1} \}$ and $g^n =1$.

The zero divisor problem has been studied in \cite{Cohen81, Linnell91, Linnell92, LinnellPuls01, Puls98}. Using results obtained in some of these papers we can give some results concerning (\ref{eq:Pompeiu}) for specific groups $G$ and classes of functions on $G$. For example, in the case $G = \mathbb{Z}^n$, where $ 2 \leq n \in \mathbb{N}$, if $f$ satisfies (\ref{eq:Pompeiu}) and $f \in \ell^p(\mathbb{Z}^n)$ for $p \leq \frac{2n}{n-1}$, then by \cite[Theorem 2.1]{LinnellPuls01} $f = 0$. Let $F_k$ denote the free group on $k$ generators, where $k \geq 2$. The main result of \cite{Linnell92} shows that $f = 0$ is the only element of $\ell^2(F_k)$ for which (\ref{eq:Pompeiu}) is true. In Section \ref{freegroupspompeiu} we will see that for $p > 2$, there exists a nonzero $f$ in $\ell^p(F_k)$ that satisfies (\ref{eq:Pompeiu}) for certain subsets of $F_k$ that are circular. 

\section{Statement of main results}\label{statementmainresults}
We start by assuming $G$ is a discrete abelian group with $r_0(G) < 2^{\omega}$. Let $M$ be a maximal ideal in $\mathbb{C}G$. Then $\mathbb{C}G/M$ is a field since $\mathbb{C}G$ is a commutative ring with identity 1. Let $i_M \colon \mathbb{C}G \mapsto \mathbb{C}G/M$ be the natural map. It was shown in the proof of Theorem 3 in \cite{laczkovichgszekelyhidi04} that $\mathbb{C}G/M$ is isomorphic to $\mathbb{C}$. Thus $i_M$ can be considered a ring homomorphism from $\mathbb{C}G$ onto $\mathbb{C}$ with kernel $M$. We define a function $f_M \colon G \rightarrow \mathbb{C}$ by
\[ f_M(g) = i_M(g) \]
for a given maximal ideal $M$ of $\mathbb{C}G$. Each $g \in G$ is invertible in $\mathbb{C}G$. Thus $f_M(g) \neq 0$ since $i_M(g) \neq 0$ in $\mathbb{C}G/M$. For $\alpha \in \mathbb{C}G$, define 
\[ Z(\alpha) \colon = \{ i_M \mid i_M (\alpha) = 0 \}, \]
where $M$ is a maximal ideal in $\mathbb{C}G$. For an ideal $I$ in $\mathbb{C}G$, set $Z(I) = \bigcap_{\alpha \in I} Z(\alpha).$ 

The main result of Section \ref{abelianpompeiu} is
\begin{Thm} \label{abelianchar} 
Let $G$ be a discrete abelian group with $r_0(G) < 2^{\omega}$ and let $\mathcal{K}$ be a finite set of finite subsets of $G$. Then $\mathcal{K}$ has the Pompeiu property if and only if $\bigcap_{K \in \mathcal{K}} Z(\chi_K) = \emptyset$.
\end{Thm}

We now shift our focus to the free group $F_k$ on $k$ generators, where $2 \leq k \in \mathbb{N}$. Indicate by $e$ the identity element of $F_k$. Any element $x$ of $F_k$ has an unique expression as a finite product of generators and their inverses, which does not contain any two adjacent factors $ww^{-1}$ or $w^{-1}w$. The number of factors in $x$ is called the {\em length} of $x$ and is denoted by $\vert x \vert$. The identity element $e$ will have length zero.

A function in $\mathcal{F}(F_k)$ is said to be {\em radial} if its value depends only on $\vert x \vert$. Let $E_n = \{ x \in F_k\mid \vert x \vert = n \}$, and let $e_n$ indicate the cardinality of $E_n$. Then $e_n = 2k(2k-1)^{n-1}$ for $n \geq 1$, and $e_0 = 1$. We shall say that a subset $A$ of $F_k$ is a {\em radial set} if $A$ is the union of sets of the form $E_n$. Let $\chi_n$ denote the characteristic function on $E_n$, so as an element of $\mathbb{C}F_k$ we have $\chi_n = \sum_{\vert x \vert = n} x$. Then every radial function has the form $\sum_{n=0}^{\infty} a_n \chi_n$, where $a_n \in \mathbb{C}$. Let $(\mathbb{C}F_k)_r$ denote the radial functions contained in $\mathbb{C}F_k$ and let $\omega = \sqrt{2k-1}$. It was shown in \cite[Chapter 3]{Figa_Picardello83} that
\begin{align*}
\chi_1 \ast \chi_1 & = \chi_2 + 2k \ast \chi_0 \\
\chi_1 \ast \chi_n & = \chi_{n+1} + \omega^2 \chi_{n-1}, \quad n \geq 2,
\end{align*}
hence $(\mathbb{C}F_k)_r$ is a commutative ring which is generated by $\chi_0$ and $\chi_1$. 

For $0 \leq n \in \mathbb{Z}$ define polynomials $p_n(z)$ by 
\begin{eqnarray}\label{eq:recpolynomials}
&p_0(z) =1, \quad p_1(z) = z, \quad p_2(z) = z^2 -2k  \nonumber \\
\text{and }                                                                    
&p_{n+1}(z) = zp_n(z) - \omega^2p_{n-1}(z) \text{ for }n \geq 2. \label{eq:recpolynomials}
\end{eqnarray}

For $\alpha = \sum_{j=0}^m a_j \chi_j \in (\mathbb{C}F_k)_r$, define a function $\hat{\alpha} \colon \mathbb{C} \rightarrow \mathbb{C}$ by
\[ \hat{\alpha} (z) =\sum_{j =0}^m a_j p_j(z). \]
Let $Z(\alpha) = \{ z \mid \hat{\alpha}(z) =0\}$. If $I$ is an ideal in $(\mathbb{C}F_k)_r$, then set $Z(I) = \bigcap_{\alpha \in I} Z(\alpha).$

In Section \ref{freegroupspompeiu} we will prove the following result, which is analogous to Theorem \ref{abelianchar}.
\begin{Thm}\label{pompeiufreechar} 
Let $2 \leq k \in \mathbb{N}$ and let $\mathcal{K}$ be a finite set of nonempty finite radial subsets of $F_k$. The $\mathcal{K}$ has the Pompeiu property if and only if $\bigcap_{K \in \mathcal{K}} Z(\chi_K) = \emptyset$.
\end{Thm}
Theorems \ref{abelianchar} and \ref{pompeiufreechar} are consequences of Theorems \ref{wienerabelian} and \ref{wienertypefree} respectively, both of which are Wiener type results.

After a moment's reflection, we see that for free groups $F_k$ the definition of $S_n(x)$ given by (\ref{eq:defsphere}) is equivalent to 
\[ S_n(x) = \{ w = xy \mid \mbox{ for some } y \in E_n \}, \]
where $x \in F_k$.

The Laplacian on $F_k$ is defined by 
\[ \triangle f(x) = \frac{1}{2k} \sum_{y \in E_1} f(xy), \]
where $x \in F_k.$ A function $f$ on $F_k$ is said to be {\em harmonic} if $\triangle f(x) =0$ for all $x \in F_k$. If $f$ is harmonic on $F_k$, then $f$ satisfies the mean-value property for all natural numbers $n$; that is, $e_n f(x) = \sum_{y \in E_n} f(xy)$ for all $n \in \mathbb{N}$ and for all $x \in F_k$. The converse to this statement is the free group version of Problem \ref{discretetworadharm}. In Section \ref{meanvalue} we give a sufficient condition that guarantees a positive answer to Problem \ref{discretetworadharm}. More specifically we prove
\begin{Thm}\label{sufficienttwocirc}
Let $n,m \in \mathbb{N}$ and suppose the only common solution to 
\[ \frac{p_n(z)}{e_n} = 1 \mbox{ and } \frac{p_m(z)}{e_m} = 1 \]
is $z = 2k.$
Then $f \in \mathcal{F}(F_k)$ is harmonic if and only if $\sum_{ y \in E_n} f(xy) = e_n f(x)$ and $\sum_{y \in E_m} f(xy) = e_m f(x)$ for all $x \in F_k$.
\end{Thm}

\section{Proof of Theorem \ref{abelianchar}}\label{abelianpompeiu}
In this section $G$ will always be a discrete abelian group with $r_0(G) < 2^{\omega}$. Recall that if $M$ is a maximal ideal in $\mathbb{C}G$, then $f_M(g) = i_M(g)$, where $i_M \colon G \to \mathbb{C}G/M$ is the natural map. We start with

\begin{Prop} \label{killedbyfm} 
Let $G$ be a discrete abelian group with $r_0(G) < 2^{\omega}$. Let $\alpha = \sum_{g \in G} a_g g \in \mathbb{C}G$ and let $M$ be a maximal ideal in $\mathbb{C}G$. Then $\alpha \in M$ if and only if $\alpha \ast \widetilde{f_M} = 0$.
\end{Prop}

\begin{proof}
Suppose $\alpha \in M$. Then $i_M (\alpha) = 0$. Because $L_x (\alpha) = x^{-1} \ast \alpha$ for $x \in G$ and $M$ is an ideal, $i_M \in Z(L_x (\alpha))$ for each $x \in G$. We now obtain
\[ 0 = i_M (L_x (\alpha)) = \sum_{g \in G} a_{xg} i_M(g) = \langle L_x (\alpha), \overline{f_M} \rangle \]
for all $x \in G$. By Lemma \ref{zerodivtrans} $\alpha \ast \widetilde{f_M} = 0$.

Conversely, let $\alpha \in \mathbb{C}G$ and assume that $\alpha \ast \widetilde{f_M} = 0$. Lemma \ref{zerodivtrans} says that $\langle L_x (\alpha), \overline{f_M} \rangle = 0$ for all $x \in G$. In particular
\[ \langle \alpha, \overline{f_M} \rangle = \sum_{g \in G} a_g i_M(g) = i_M(\alpha). \]
Thus, $\alpha \in M$.
\end{proof}

We now state and prove a Wiener type result.
\begin{Thm} \label{wienerabelian}
Let $G$ be a discrete abelian group with $r_0(G) < 2^{\omega}$ and let $I$ be an ideal in $\mathbb{C}G$. Then $Z(I) = \emptyset$ if and only if $I = \mathbb{C}G$. 
\end{Thm}

\begin{proof}
Assume $Z(I) \neq \emptyset$. Then there exists a maximal ideal $M$ in $\mathbb{C}G$ for which $i_M (\alpha) = 0$ for each $\alpha \in I$. The kernel of $i_M$ is precisely $M$, so $\alpha \in M$. By Proposition \ref{killedbyfm}, $\alpha \ast \widetilde{f_M} = 0$ for all $\alpha \in I$. Thus $1 \notin I$ because $1 \ast \widetilde{f_M} = \widetilde{f_M}$ and $\widetilde{f_M}  \neq 0$. Hence, $I \neq \mathbb{C}G.$

Now suppose $I \neq \mathbb{C}G$. Let $M$ be a maximal ideal that contains $I$. Then $i_M(\alpha) = 0$ for each $\alpha \in I$. Hence $i_M \in Z(I)$.
\end{proof}

We now prove Theorem \ref{abelianchar}. Assume $\bigcap_{K \in \mathcal{K}} Z(\chi_K) = \emptyset$. Let $I$ be the ideal in $\mathbb{C}G$ generated by the functions $\chi_K, K \in \mathcal{K}$. Then $Z(I) = \emptyset$. Suppose that $\mathcal{K}$ does not have the Pompeiu property. Proposition \ref{pompeiuzerodiv} says that there exists a nonzero function $f \in \mathcal{F}(G)$ for which $\chi_K \ast \tilde{f} = 0$ for all $K \in \mathcal{K}$. Thus $I \neq \mathbb{C}G$. Using Theorem \ref{wienerabelian} we obtain the contradiction $Z(I) \neq \emptyset$. Thus $\mathcal{K}$ has the Pompeiu property. 

Conversely, assume $\bigcap_{K \in \mathcal{K}} Z(\chi_K) \neq \emptyset$. Then there exists a maximal ideal $M$ in $\mathbb{C}G$ such that $i_M \in Z(\chi_K)$ for each $K \in \mathcal{K}$. Thus $\chi_K \in M$ for each $K \in \mathcal{K}$. By combining Proposition \ref{killedbyfm} and Proposition \ref{pompeiuzerodiv} we obtain that $\mathcal{K}$ does not have the Pompeiu property. The proof of Theorem \ref{abelianchar} is complete.

\section{Proof of Theorem \ref{pompeiufreechar}} \label{freegroupspompeiu}
Let $F_k$ denote the free group on $k$ generators, where $2 \leq k \in \mathbb{N}$. We start the section by giving some definitions and results that will be needed for the proof of Theorem \ref{pompeiufreechar}. After giving these results we will prove Theorem \ref{pompeiufreechar}.  We will conclude the section with a discussion about how our work is connected with some of the results in \cite{CohenPicardello88}.

A slight modification to the proof of \cite[Lemma 6.1]{LinnellPuls01} gives the following elementary result, which will be needed later.
\begin{Lem} \label{elementary}
Let $x,y \in F_k$, with $\vert x \vert = \vert y \vert$, and let $0 \leq m,n \in \mathbb{Z}$. Then 
\[ \langle x \ast \chi_m, \chi_n \rangle = \langle y \ast \chi_m, \chi_n \rangle. \]
\end{Lem}

A function $\psi$ on $F_k$ is said to be {\em spherical} if and only if it defines a multiplicative linear functional on $(\mathbb{C}F_k)_r$. See pages 33-34, especially Lemma 1.5, of \cite{Figa_Picardello83} for a more detailed discussion about spherical functions on $F_k$. Let $z \in \mathbb{C}$ and define $\phi_z \in \mathcal{F}(F_k)$ by 
\[ \phi_z = \sum_{n=0}^{\infty} \frac{p_n(z)}{e_n} \chi_n. \]
Now,
\begin{align*}
\phi_z \ast \frac{\chi_1}{2k}  & =  \frac{p_1(z)}{e_1}\chi_0 + \sum_{n=1}^{\infty} \left( \frac{p_{n-1} (z)}{2ke_{n-1}} + \frac{\omega^2 p_{n+1}(z)}{2ke_{n+1}}\right)\chi_n \\
     & =  \frac{z}{2k} \sum_{n=0}^{\infty} \frac{p_n(z)}{e_n} \chi_n \\
     & =  \frac{z}{2k} \phi_z.
\end{align*}

By \cite[Corollary 1.7]{Figa_Picardello83} $\phi_z$ is spherical, equivalently it is a multiplicative linear functional on $(\mathbb{C}F_k)_r$. We can now prove 
\begin{Prop} \label{radialspan}
Let $\alpha \in (\mathbb{C}F_k)_r$ and let $z \in \mathbb{C}$. Then $\alpha \ast \phi_z = 0$ if and only if $\langle \chi_n \ast \alpha, \overline{\phi_z} \rangle = 0$ for all integers $n \geq 0$.
\end{Prop}
\begin{proof}
Using $\chi_n \ast \alpha = \sum_{g \in E_n} L_g \alpha$ with Lemma \ref{zerodivtrans}, $\langle \chi_n \ast \alpha, \overline{\phi_z} \rangle = 0$ for each $n \geq 0.$

Conversely, write $\alpha = \sum_{j=0}^m a_j \chi_j$. Then 
\begin{align*}
\langle \alpha, \overline{\phi_z} \rangle & = \sum_{j,n} \frac{a_j p_n(z)}{e_n} \langle \chi_j, \chi_n \rangle \\
                                          & = \sum_{j=0}^m a_j p_j(z) = \hat{\alpha}(z).
\end{align*}
The function $\overline{\phi_z}$ is a multiplicative linear functional on $(\mathbb{C}F_k)_r$ because $\overline{\phi_z} = \phi_{\overline{z}}$. Thus $\langle \chi_n \ast \alpha, \overline{\phi_z} \rangle = \hat{\alpha}(z)p_n(z).$ Lemma \ref{elementary} allows us to deduce that if $x \in F_k$ and $\vert x \vert = n$, then $\langle x \ast \alpha, \overline{\phi_z} \rangle = \hat{\alpha}(z) \frac{p_n(z)}{e_n}.$ The result now follows from Lemma \ref{zerodivtrans}.
\end{proof}

\begin{Cor}\label{containzeros}
Let $0 \neq \alpha \in (\mathbb{C}F_k)_r$ and let $z \in \mathbb{C}$. Then $\alpha \ast \phi_z = 0$ if and only if $z \in Z(\alpha).$
\end{Cor}
\begin{proof} 
We saw in the proof of the above proposition that $\langle \chi_n \ast \alpha, \overline{\phi_z} \rangle = \hat{\alpha}(z) p_n(z)$ for each positive integer $n$. The corollary now follows.
\end{proof}

We need the following fact about maximal ideals in $(\mathbb{C}F_k)_r$.
\begin{Prop} \label{maximalideals}
If $M$ is a maximal ideal in $(\mathbb{C}F_k)_r$, then $(\mathbb{C}F_k)_r / M$ is isomorphic to $\mathbb{C}$.
\end{Prop}
\begin{proof}
Let $\mathbb{C}[x]$ be the polynomial ring over $\mathbb{C}$. The proposition will be proved if we can show $(\mathbb{C}F_k)_r$ is isomorphic to $\mathbb{C}[x]$, since $\mathbb{C}[x] / J$ is isomorphic to $\mathbb{C}$ for each maximal ideal $J$ in $\mathbb{C}[x]$. For $\alpha = \sum_{i=0}^n a_i \chi_i \in (\mathbb{C}F_k)_r$, define $T\alpha \in \mathbb{C}[x]$ by 
\[ T \alpha = \sum_{i=0}^n a_i p_i(x). \]
We will now show that $T$ gives the desired isomorphism. Clearly $T$ is linear. Note that if $z \in \mathbb{C}$, then $T\alpha (z) = \langle \phi_z, \alpha \rangle$. Combining this observation with the fact $\phi_z$ is multiplicative on $(\mathbb{C}F_k)_r$ for each $z \in \mathbb{C}$, we obtain $T(\alpha \ast \beta) = (T\alpha)(T\beta)$, where $\beta \in (\mathbb{C}F_k)_r$. If $T \alpha = \sum_{i=0}^n a_i p_i(x) = 0$, then it must be the case $a_0 = a_1 = \cdots = a_n =0$. Thus $T$ is one-to-one. Let $q \in \mathbb{C}[x]$. We may assume without loss of generality that $q$ is a monic polynomial. Write $q = (x - c_1)(x-c_2) \cdots (x-c_n)$, where $c_k \in \mathbb{C}$ and $n$ is the degree of $q$. The function $(\chi_1 - c_1) \ast (\chi_1 - c_2) \ast \cdots \ast (\chi_1 - c_n) \in (\mathbb{C}F_k)_r$ and $T$ maps this element to $q$, so $T$ is onto. Hence $T$ is an isomorphism between $(\mathbb{C} F_k)_r$ and $\mathbb{C}[x]$ and the proposition now follows.
\end{proof} 
Using the polynomial relations (\ref{eq:recpolynomials}) with an induction argument we obtain 
\[ \chi_n = p_n(\chi_1). \]
We can now state and prove the following Wiener type result for $F_k$, which corresponds to Theorem \ref{wienerabelian}.

\begin{Thm}\label{wienertypefree}
Let $F_k$ be the free group with $2 \leq k \in \mathbb{N}$. Let $I$ be an ideal in $(\mathbb{C}F_k)_r$. Then $Z(I) = \emptyset$ if and only if $I = (\mathbb{C}F_k)_r$. 
\end{Thm}
\begin{proof}
Suppose $Z(I) \neq \emptyset$ and let $z \in Z(I)$. Then for each $\alpha \in I, \alpha \ast \phi_z = 0$ by Corollary \ref{containzeros}. Since $\chi_0 \ast \phi_z \neq 0$ it follows that $I \neq (\mathbb{C}F_k)_r$. 

Conversely, assume $I \neq (\mathbb{C}F_k)_r$. Let $M$ be a maximal ideal in $(\mathbb{C}F_k)_r$ that contains $I$. Then $(\mathbb{C} F_k)_r / M$ is isomorphic to $\mathbb{C}$. Hence, there is a ring homomorphism $\psi \colon (\mathbb{C}F_k)_r \mapsto \mathbb{C}$ for which the kernel of $\psi$ is $M$. The multiplicative linear functional $\psi$ is determined by its value on $\chi_1$. Set $z = \psi(\chi_1)$. Because $\chi_n = p_n (\chi_1)$, we obtain $\psi(\chi_n) = p_n(\psi(\chi_1)) = p_n(z)$. Consequently, $\psi = \phi_z$. Furthermore, if $\alpha \in I$, then $\langle \chi_n \ast \alpha, \overline{\phi_z} \rangle = 0$ for each nonnegative integer $n$. By Proposition \ref{radialspan} and Corollary \ref{containzeros}, $z \in Z(\alpha)$ for each $\alpha \in I$. Therefore, $z \in Z(I)$.
\end{proof}

We now prove Theorem \ref{pompeiufreechar}. First assume $\bigcap_{K \in \mathcal{K}} Z(\chi_K) \neq \emptyset$ and let $z \in \bigcap_{K \in \mathcal{K}} Z(\chi_K)$. Then $\chi_K \ast \phi_z = 0$ for each $K \in \mathcal{K}$. Since $\chi_K$ and $\phi_z$ are both radial, it follows from Proposition \ref{pompeiuzerodiv} that $\mathcal{K}$ does not have the Pompeiu property.

Now assume $\bigcap_{K \in \mathcal{K}} Z(\chi_K) = \emptyset$. Let $I$ be the ideal in $(\mathbb{C}F_k)_r$ generated by $\{ \chi_K \}_{K \in \mathcal{K}}$. Then $Z(I) = \emptyset$ and by Theorem \ref{wienertypefree}, $I = (\mathbb{C}F_k)_r$. If $f \in \mathcal{F}(F_k)$ and $ \chi_K \ast f = 0$ for each $K \in \mathcal{K}$, then $\chi_0 \ast f = 0$. Hence, it must the case $f = 0$ and $\mathcal{K}$ has the Pompeiu property. The proof of Theorem \ref{pompeiufreechar} is now complete.

A special case of the Pompeiu problem is the two circle problem. 

\begin{Prob} \label{twocircles}
Let $r$ and $s$ be natural numbers. Is $f = 0$ the only function on $G$ that satisfies
\begin{equation}
\sum_{d(x, y) =r} f(y) = 0 = \sum_{d(x,y) = s} f(y)  \label{eq:twocirclecond}
\end{equation}
for all $x \in G$?
\end{Prob}

Cohen and Picardello studied this problem in \cite{CohenPicardello88} for homogeneous trees. The Cayley graph of the free group $F_k, 2 \leq k \in \mathbb{N}$, is a homogeneous tree of degree $2k$. We now proceed to explain how our results overlap with some of the results in \cite{CohenPicardello88}.

Let $K_n$ be the set of elements in $F_k$ that have length $n$. Each element of $K_n$ has distance $n$ from the identity element $e$ of $F_k$. From a geometric point of view we can think of $K_n$ as a sphere in $F_k$ of radius $n$ with center $e$. Denote the characteristic function on $K_n$ by $\chi_n$. Set
\[ \phi_0 = \sum_{n=0}^{\infty} \frac{p_n(0)}{e_n} \chi_n = \sum_{n=0}^{\infty} \frac{(-1)^n}{(2k-1)^n} \chi_{2n}. \]
Since $\{0\} = Z(\chi_1)$, Corollary \ref{containzeros} says that $\chi_1 \ast \phi_0 =0$. Furthermore, $\phi_0 \in \ell^p(F_k)$ for $p > 2$. Because $\chi_1$ and $\phi_0$ are both radial, $\chi_1 \ast \widetilde{\phi_0} = \chi_1 \ast \phi_0 = 0$. Thus $\phi_0$ satisfies (\ref{eq:Pompeiu}) by Proposition \ref{pompeiuzerodiv}. So $K_1$ does not have the Pompeiu property when the class of functions under consideration contains $\ell^p(F_k)$ for some $p > 2$. In fact, if $n$ an odd natural number, then $0 \in Z(\chi_n)$ due to $\widehat{\chi_n} = p_n(z)$ being an odd polynomial. Thus $\chi_n \ast \widetilde{\phi_0} = 0$ and $K_n$ does not have the Pompeiu property for all odd natural numbers. We now move on to the two circle problem. 

Let $r$ and $s$ be natural numbers and set $\mathcal{K} = \{K_r, K_s\}$. Observe that if $x \in F_k$, then $xK_r$ and $xK_s$ are spheres in $F_k$ with center $x$ and radius $r$ and $s$ respectively. Consequently, $f=0$ is the only solution to (\ref{eq:twocirclecond}) if and only if $\mathcal{K}$ has the Pompeiu property. Theorem \ref{pompeiufreechar} tells us $\mathcal{K}$ has the Pompeiu property precisely when $\widehat{\chi_s} = p_s(z)$ and $\widehat{\chi_r} = p_r(z)$ do not have a common root. If $r$ and $s$ are both odd, then as we saw in the previous paragraph $0$ is a common root of both $\widehat{\chi_s}$ and $\widehat{\chi_r}$. Thus $\mathcal{K}$ does not have the Pompeiu property and $\phi_0$ is a nonzero function that satisfies (\ref{eq:twocirclecond}). This is essentially Theorem 1(a) of \cite{CohenPicardello88} for the case of homogeneous trees with positive even degree.

Now suppose $r$ and $s$ are both not odd. It was shown in \cite[Section 2]{CohenPicardello88} that $\widehat{\chi_r} = p_r(z)$ and $\widehat{\chi_s} = p_s(z)$ do not have a common root. Thus $\mathcal{K}$ has the Pompeiu property and $f =0$ is the only solution to (\ref{eq:twocirclecond}).

\section{Proof of Theorem \ref{sufficienttwocirc}} \label{meanvalue}
Recall that $F_k$ is the free group on $k$ generators, where $k \geq 2$. Before we prove Theorem \ref{sufficienttwocirc} we need to give several preliminary results. We start with the following crucial lemma, which connects the mean-value property with the problem of zero divisors on $F_k$.
\begin{Lem} \label{connecttozerodiv}
Let $f$ be a complex-valued function on $F_k$ and let $n \in \mathbb{N}$. Then $\sum_{y \in E_n} f(xy) = e_n f(x)$ for all $x \in F_k$ if and only if $f \ast (\chi_n - e_n) =0$.
\end{Lem}
\begin{proof}
For $x \in F_k$,
\begin{eqnarray*}
f \ast \chi_n (x)   &  =  &  \sum_{y \in F_k} f(xy^{-1}) \chi_n(y)   \\
                             &   = & \sum_{y \in E_n} f(xy^{-1})   \\
                            &   =  &   \sum_{y \in E_n} f(xy). 
\end{eqnarray*}
Hence, $\sum_{y \in E_n} f(xy) = e_n f(x)$ for all $x \in F_k$ if and only if $f \ast (\chi_n - e_n) = 0.$
\end{proof}

Define a linear map $P \colon \mathcal{F}(F_k) \rightarrow \mathcal{F}_r(F_k)$ by
\[ P(f) (x) = \frac{1}{e_n} \sum_{w \in E_n} f(w), \]
where $\vert x \vert = n.$

Our next two results show that in order to prove Theorem \ref{sufficienttwocirc} we will only need to consider radial functions.
\begin{Lem} \label{factorthrough}
Let $f \in \mathcal{F}(F_k)$. Then $P(f \ast \chi_n) = P(f) \ast \chi_n$ for all $n \in \mathbb{N}.$
\end{Lem}
\begin{proof}
Let $n = 1$ and let $x \in F_k$. Suppose $\vert x \vert = m$. Then 
\begin{eqnarray*}
P(f \ast \chi_1) (x)  &  =  &  \frac{1}{e_m} \sum_{w \in E_m} f \ast \chi_1 (w) \\
                                 & =  & \frac{1}{e_m} \sum_{w \in E_m} \sum_{y \in F_k} f(wy^{-1}) \chi_1 (y) \\
                                & =  & \frac{1}{e_m} \sum_{w \in E_m} \sum_{y \in E_1} f(wy).  
\end{eqnarray*}
If $w \in E_m$ and $y \in E_1$, then $\vert wy \vert = m+1$ or $m-1$. For each $g \in E_{m+1}$, there is exactly one element $w \in E_m$ and one element $y \in \chi_1$ such that $g = wy$. If $g \in E_{m-1}$, then there is exactly $2k-1$ elements $w \in E_m$ such that $g = wy$ for some $y \in \chi_1.$ Consequently,
\begin{eqnarray*}
\frac{1}{e_m} \sum_{w \in E_m} \sum_{y \in E_1} f(wy) &  =  &  \frac{1}{e_m} \left( \sum_{w \in E_{m+1}} f(w) + (2k-1) \sum_{w \in E_{m-1} }f(w) \right) \\
                                 & = & \frac{2k-1}{e_{m+1}} \sum_{w \in E_{m+1}} f(w) + \frac{1}{e_{m-1}} \sum_{w \in E_{m-1}} f(w) \\
                                & = & \sum_{y \in E_1} P(f)(xy) \\
                                & = &  P(f) \ast \chi_1 (x). 
\end{eqnarray*}
We now consider the case $ n = 2$.
\begin{eqnarray*}
P(f \ast \chi_2)  & = & P(f \ast \chi_1 \ast \chi_1 - 2k f ) \\
                            & = & P(f \ast \chi_1) \ast \chi_1 - 2k P(f) \\
                           &  =  &  P(f) \ast \chi_1 \ast \chi_1 - 2k P(f) \\
                          & =   & P(f) \ast ( \chi_1 \ast \chi_1 - 2k \chi_0) \\
                           & = &  P(f) \ast \chi_2.
\end{eqnarray*}
It now follows from induction that $P(f \ast \chi_n) = P(f) \ast \chi_n$ for all $n \in \mathbb{N}$.
\end{proof}

\begin{Prop} \label{reducetoradial}
If $f \ast (\chi_n - e_n) =0$, then $P(f) \ast (\chi_n - e_n) =0.$
\end{Prop}
\begin{proof}
Let $x \in F_k$ and assume $\vert x \vert = m$. By using Lemma \ref{factorthrough} and Lemma \ref{connecttozerodiv}, we obtain
\begin{eqnarray*}
P(f) \ast \chi_n (x) & = & \frac{1}{e_m} \sum_{w \in E_m} f\ast \chi_n (w) \\
                            & = &  \frac{1}{e_m} \sum_{w \in E_m} \left( \sum_{y \in E_n} f(wy) \right) \\
                            & = & \frac{1}{e_m} \sum_{w \in E_m} e_n f(w) \\
                            & = & \frac{e_n}{e_m} \sum_{w\in E_m} f(w) \\
                           & = & e_n P(f)(x).
\end{eqnarray*}
Thus, $P(f) \ast (\chi_n - e_n) =0.$
\end{proof}

Our next result shows that $z = 2k$ is indeed a simple root of $p_n(z) -e_n =0$ for all $n \in \mathbb{N}.$

\begin{Lem}\label{simpleroot} 
For all nonnegative integers $n, z =2k$ is a root of multiplicity one for $p_n(z) - e_n = 0$.
\end{Lem}
\begin{proof}
Clearly $p_0(2k) = e_0$ and $p_1(2k) = e_1$. Let $n \geq 1$ and assume $p_m(2k) = e_m$ for all $m \leq n$. Then,
\begin{eqnarray*}
p_{n+1}(2k)   &  =  &  2kp_n(2k) - (2k-1) p_{n-1} (2k) \\
                        & =   &  (2k-1)^{n-1} (2k (2k-1))    \\
                       & =   &  e_{n+1},
\end{eqnarray*}
so $z = 2k$ is a root of $p_n(z) - e_n = 0$ for all $n \in \mathbb{N}.$ 

We will now show that $p_{n+1}'(2k) \geq p_n'(2k)$ for all $n \geq 1$. This will establish that $z = 2k$ is a root of multiplicity one for $p_n(z) -e_n =0$ because $p_1'(2k) = 1$ and $(p_n - e_n)'(2k) = p_n'(2k)$ for all $n \geq 1$. Obviously, $p_2'(2k) \geq p_1'(2k).$ For some $ n \geq 2$ assume $p_n' (2k) \geq p_{n-1}'(2k).$ Thus 
\begin{eqnarray*}
p_{n+1}' (2k) & = & p_n(2k) + 2k p_n'(2k) - (2k-1) p_{n-1}'(2k)  \\
                       & > & 2k p_n'(2k) - (2k-1) p_{n-1}'(2k) \\
                       & \geq & 2k p_n'(2k) - (2k-1) p_n'(2k)  \\
                       & = & p_n'(2k). 
\end{eqnarray*}
Hence, $p_{n+1}'(2k) > p_n'(2k)$ for all $n \geq 1$. Thus $z = 2k$ is a root of order one for $p_n(z) - e_n = 0$. 
\end{proof}

We now prove Theorem \ref{sufficienttwocirc}. Assume there exists distinct natural numbers $n$ and $m$ such that $f \ast (\chi_n - e_n) = 0 = f \ast (\chi_m - e_m)$ for some nonzero $f \in \mathcal{F}(F_k)$. There exists an $x \in F_k$ such that $f(x) \neq 0.$ Replacing $f$ by $x^{-1} \ast f$ if necessary, we may assume $f(e) \neq 0$. Hence, $P(f) \neq 0.$ By Proposition \ref{reducetoradial} we can and do assume that $f \in \mathcal{F}_r (F_k)$. By Lemma \ref{simpleroot}, $2k \in Z(\chi_n - e_n) \cap Z(\chi_m - e_m)$. In the proof of Proposition \ref{maximalideals} it was shown there was an isomorphism between $(\mathbb{C}F_k)_r$ and $\mathbb{C}[x]$. It now follows that there exists $\alpha_1, \alpha_2 \in (\mathbb{C}F_k)_r$ for which 
\[ \chi_n - e_n = (\chi_1 - 2k)\ast \alpha_1 \mbox{ and } \chi_m - e_m = (\chi_1 - 2k)\ast \alpha_2. \]
Lemma \ref{simpleroot} implies $2k \notin Z(\alpha_1)$ and $2k \notin Z(\alpha_2).$  It now follows from the hypothesis that $Z(\alpha_1) \cap Z(\alpha_2) = \emptyset$. Let $I$ be the ideal in $(\mathbb{C}F_k)_r$ generated by $\alpha_1$ and $\alpha_2$. Theorem \ref{wienertypefree} says that $I = (\mathbb{C}F_k)_r$ since $Z(I) = \emptyset$. In particular, $\chi_0 \in I$. Consequently, $\chi_1 - 2k$ is an element of the ideal generated by $\chi_n - e_n = (\chi_1 - 2k) \ast \alpha_1$ and $\chi_m - e_m = (\chi_1 - 2k)\ast \alpha_2$. Therefore, $(\chi_1 - 2k) \ast f = 0$ and $f$ is harmonic.

Conversely, assume $f \in \mathcal{F}(F_k)$ is harmonic. Since $\chi_1 - 2k$ is a factor of $\chi_n - e_n$, it follows immediately that $f\ast (\chi_n - e_n) = 0$ for all $n \geq 1$. The proof of Theorem \ref{sufficienttwocirc} is now complete.

An interesting question is the following: What pairs of natural numbers $n$ and $m$ satisfy the hypothesis of Theorem \ref{sufficienttwocirc}? The proof of Lemma \ref{simpleroot} also demonstrates that $p_n(-2k) - e_n =0$ for all even $n \in \mathbb{N}$. This is also true because $p_n(z) - e_n$ is an even polynomial for all even $n$. Now 
\[ \phi_{-2k} (x) = \left\{ \begin{array}{rl}
1      &  \vert x \vert \mbox{ is even} \\
-1    &  \vert x \vert \mbox{ is odd. }\end{array} \right. \]
For even $n \in \mathbb{N}, -2k \in Z(\chi_n - e_n)$ since $\widehat{\chi_n - e_n} (-2k) = p_n(-2k) - e_n =0$. It now follows from Corollary \ref{containzeros} that if both $n$ and $m$ are even, then
\[ \phi_{-2k} \ast (\chi_n - e_n) = 0 = \phi_{-2k} \ast (\chi_m - e_m). \]
However, $\phi_{-2k}$ is not harmonic because $\phi_{-2k} \ast (\chi_1 - 2k) \neq 0$. Thus the hypothesis of Theorem \ref{sufficienttwocirc} are not satisfied if $n$ and $m$ are both even.

Equation (\ref{eq:contmeanvalue}) was studied in the setting of homogeneous trees in \cite{GonzalezVieli98}. This paper is relevant here since the Cayley graph of $F_k$ is a homogeneous tree of degree $2k$. The main result in \cite{GonzalezVieli98} shows that the hypothesis of Theorem \ref{sufficienttwocirc} is true when $n$ and $m$ are relatively prime. All this leads us to ask if the hypothesis of Theorem \ref{sufficienttwocirc} is true when $n$ and $m$ are both not even?
\bibliographystyle{plain}
\bibliography{mvppompeiudiscretegroups}

\begin{thebibliography}{10}

\bibitem{BerensteinZalcman80}
Carlos~A. Berenstein and Lawrence Zalcman.
\newblock Pompeiu's problem on symmetric spaces.
\newblock {\em Comment. Math. Helv.}, 55(4):593--621, 1980.

\bibitem{CareyKaniuthMoran91}
Alan~L. Carey, Eberhard Kaniuth, and William Moran.
\newblock The {P}ompeiu problem for groups.
\newblock {\em Math. Proc. Cambridge Philos. Soc.}, 109(1):45--58, 1991.

\bibitem{Cohen81}
Joel~M. Cohen.
\newblock von {N}eumann dimension and the homology of covering spaces.
\newblock {\em Quart. J. Math. Oxford Ser. (2)}, 30(118):133--142, 1979.

\bibitem{CohenPicardello88}
Joel~M. Cohen and Massimo~A. Picardello.
\newblock The {$2$}-circle and {$2$}-disk problems on trees.
\newblock {\em Israel J. Math.}, 64(1):73--86, 1988.

\bibitem{Figa_Picardello83}
Alessandro Fig{\`a}-Talamanca and Massimo~A. Picardello.
\newblock {\em Harmonic analysis on free groups}, volume~87 of {\em Lecture
  Notes in Pure and Applied Mathematics}.
\newblock Marcel Dekker Inc., New York, 1983.

\bibitem{GonzalezVieli98}
Francisco~Javier Gonz{\'a}lez~Vieli.
\newblock A two-spheres problem on homogeneous trees.
\newblock {\em Combinatorica}, 18(2):185--189, 1998.

\bibitem{laczkovichgszekelyhidi04}
M.~Laczkovich and G.~Sz{\'e}kelyhidi.
\newblock Harmonic analysis on discrete abelian groups.
\newblock {\em Proc. Amer. Math. Soc.}, 133(6):1581--1586 (electronic), 2005.

\bibitem{Linnell91}
P.~A. Linnell.
\newblock Zero divisors and group von {N}eumann algebras.
\newblock {\em Pacific J. Math.}, 149(2):349--363, 1991.

\bibitem{Linnell92}
Peter~A. Linnell.
\newblock Zero divisors and {$L^2(G)$}.
\newblock {\em C. R. Acad. Sci. Paris S\'er. I Math.}, 315(1):49--53, 1992.

\bibitem{LinnellPuls01}
Peter~A. Linnell and Michael~J. Puls.
\newblock Zero divisors and {$L^p(G)$}. {II}.
\newblock {\em New York J. Math.}, 7:49--58 (electronic), 2001.

\bibitem{PeyerimhoffSamior10}
Norbert Peyerimhoff and Evangelia Samiou.
\newblock Spherical spectral synthesis and two-radius theorems on
  {D}amek-{R}icci spaces.
\newblock {\em Ark. Mat.}, 48(1):131--147, 2010.

\bibitem{Puls98}
Michael~J. Puls.
\newblock Zero divisors and {$L^p(G)$}.
\newblock {\em Proc. Amer. Math. Soc.}, 126(3):721--728, 1998.

\bibitem{RawatSitaram97}
Rama Rawat and Alladi Sitaram.
\newblock The injectivity of the {P}ompeiu transform and {$L^p$}-analogues of
  the {W}iener-{T}auberian theorem.
\newblock {\em Israel J. Math.}, 91(1-3):307--316, 1995.

\bibitem{ScottSitaram88}
David Scott and Alladi Sitaram.
\newblock Some remarks on the {P}ompeiu problem for groups.
\newblock {\em Proc. Amer. Math. Soc.}, 104(4):1261--1266, 1988.

\bibitem{Zalcman80}
Lawrence Zalcman.
\newblock Offbeat integral geometry.
\newblock {\em Amer. Math. Monthly}, 87(3):161--175, 1980.

\end{thebibliography}

\end{document}